\documentclass[preprint, 12pt, english]{elsarticle}
\usepackage{amsmath}
\usepackage{latexsym, amssymb, mathtools}
\usepackage{amsthm}
\usepackage{color}
\usepackage{txfonts}

\newtheorem{thm}{Theorem}[section] 

\newtheorem{cor}[thm]{Corollary}

\newtheorem{lem}[thm]{Lemma}
\newtheorem{prop}[thm]{Proposition}

\theoremstyle{definition}
\newtheorem{rem}[thm]{Remark}

\newcommand\operA[2]{{\if!#2!\operatorname{#1}\else{\operatorname{#1}_{#2}^{\phantom{I}}}\fi}} 

\newcommand\Cref[1]{{Corollary~\ref{#1}}}%

\def\norm{{\operatorname{N}}}

\newcommand{\Trace}[1][]{\if!#1!\operatorname{Tr}\else{\operatorname{Tr}_{#1}^{\phantom{I}}}\fi} 

\long\def\forget#1\forgotten{{}} %

\def\({\left(}
\def\){\right)}

\usepackage{tabu}

\newif\iffurther
\furtherfalse

\newif\ifXY 
\XYtrue     
\ifXY

\input xy
\input xyidioms.tex
\usepackage{xy}
\xyoption{all} %
\fi 

\usepackage{babel}

\journal{Journal of Algebra and its Applications}

\begin{document}

\begin{frontmatter}

\title{Symbol Length of $p$-Algebras of Prime Exponent}

\author{Adam Chapman}
\ead{adam1chapman@yahoo.com}

\address{Department of Computer Science, Tel-Hai College, Upper Galilee, 12208 Israel}

\begin{abstract}
We prove that if the maximal dimension of an anisotropic homogeneous polynomial form of prime degree $p$ over a field $F$ with $\operatorname{char}(F)=p$ is a finite integer $d$ greater than 1 then the symbol length of $p$-algebras of exponent $p$ over $F$ is bounded from above by $\left \lceil \frac{d-1}{p} \right \rceil-1$, and show that every two tensor products of symbol algebras of lengths $k$ and $\ell$ with $(k+\ell) p \geq d-1$ can be modified so that they share a common slot. For $p=2$, we obtain an upper bound of $\frac{u(F)}{2}-1$ for the symbol length, which is sharp when $I_q^3 F=0$.
\end{abstract}

\begin{keyword}
Central Simple Algebras, $p$-Algebras, Symbol Length, Linkage, Homogeneous Polynomial Forms, Quadratic Forms, $u$-invariant
\MSC[2010] 16K20 (primary); 11E76, 11E81 (secondary)
\end{keyword}

\end{frontmatter}

\section{Introduction}

A $p$-algebra is a central simple algebra of degree $p^m$ over a field $F$ with $\operatorname{char}(F)=p$ for some prime number $p$ and positive integer $m$.

It was proven in \cite[Chapter 7, Theorem 30]{Albert} that every $p$-algebra of exponent $p$ is Brauer equivalent to a tensor product of algebras of the form
$$[\alpha,\beta)_{p,F}=F \langle x,y : x^p-x=\alpha, y^p=\beta, y x-x y=y \rangle$$
for some $\alpha \in F$ and $\beta \in F^\times$.
We call such algebras ``symbol algebras".
An equivalent statement was proven in \citep{MS} when $\operatorname{char}(F) \neq p$ and $F$ contains primitive $p$th roots of unity.

This means the group $\prescript{}{p}{Br}(F)$ is generated by symbol algebras.
The minimal number of symbol algebras required in order to express a given $A$ in $\prescript{}{p}{Br}(F)$ is called the symbol length of $A$.
The symbol length of $\prescript{}{p}{Br}(F)$ is the supremum of the symbol lengths of all $A \in \prescript{}{p}{Br}(F)$.
This number gives an indication of how complicated this group is, and due to the significance of this group, the symbol length has received special attention in papers such as \cite{Florence} and \cite{Matzri}.

In \cite[Theorem 3.4]{Matzri} it was proven that when $\operatorname{char}(F) \neq p$, $F$ contains primitive $p$th roots of unity and the maximal dimension of an anisotropic homogeneous polynomial form of degree $p$ over $F$ is a finite integer $d$, the symbol length of $A$ is at most $\left \lceil \frac{d+1}{p} \right \rceil-1$. The case of $p$-algebras was also considered in that paper, but only over $C_m$-fields, which involve assumptions on the greatest dimensions of anisotropic homogeneous polynomial forms of any degree, not only $p$.

We prove that when $\operatorname{char}(F)=p$ and the maximal dimension of an anisotropic form of degree $p$ over $F$ is a finite integer $d$ greater than 1, the symbol length of $A$ is at most $\left \lceil \frac{d-1}{p} \right \rceil-1$.
We obtain this bound by first proving that every two tensor products of symbol algebras $\bigotimes_{i=1}^k C_i$ and $\bigotimes_{i=1}^\ell D_i$ with $(k+\ell)p \geq d-1$ can be modified so that they share a common slot.

In the last section, we focus on the case of $p=2$. We explain how the same argument holds in this case if we replace $d$ with the $u$-invariant, i.e. the maximal dimension of an anisotropic nonsingular quadratic form over $F$, and why the upper bound $\frac{u(F)}{2}-1$ for the symbol length is sharp when $I_q^3 F=0$.

\section{Preliminaries}

We denote by $M_p(F)$ the $p \times p$ matrix algebra over $F$.
For two given central simple algebras $A$ and $B$ over $F$, we write $A=B$ if the algebras are isomorphic.
We recall most of the facts we need on symbol algebras.
For general background on these algebras see \cite[Chapter 7]{Albert}.

We start with some basic symbol changes.
In the following remark and lemmas, let $p$ be a prime integer and $F$ be a field with $\operatorname{char}(F)=p$.
We say that a homogeneous polynomial form $\varphi$ over $F$ is anisotropic if it has no nontrivial zeros.

\begin{rem}\label{rem1}
\begin{itemize}
\item[(a)] $M_p(F)=[\alpha,1)_{p,F}=[0,\beta)_{p,F}$ for every $\alpha \in F$ and $\beta \in F^\times$.
\item[(b)] $[\alpha,\beta)_{p,F} \otimes [\alpha,\gamma)_{p,F}=M_p(F) \otimes [\alpha,\beta \gamma)_{p,F}$.
\item[(c)] $[\alpha,\gamma)_{p,F} \otimes [\beta,\gamma)_{p,F}=M_p(F) \otimes [\alpha+\beta,\gamma)_{p,F}$.
\item[(d)] If $[\alpha,\beta)_{p,F}$ contains zero divisors then it is $M_p(F)$.
\item[(e)] Given $K=F[x : x^p-x=\alpha]$, we denote by $\operatorname{N}_{K/F}$ the norm form $K \rightarrow F$. This form is homogeneous of degree $p$ over $F$ of dimension $[K:F]=p$.
\end{itemize}
\end{rem}

\begin{lem}\label{addtofirst}
Consider $[\alpha,\beta)_{p,F}=F \langle x,y : x^p-x=\alpha, y^p=\beta, yx-xy=y \rangle$ for some $\alpha \in F$ and $\beta \in F^\times$.
Then 
\begin{itemize}
\item[(a)] $[\alpha,\beta)_{p,F}=[\alpha,\operatorname{N}_{F[x]/F}(f) \beta)_{p,F}$ for any $f \in F[x]$ with $\operatorname{N}_{F[x]/F}(f) \neq 0$.
\item[(b)] $[\alpha,\beta)_{p,F}=[\alpha+\beta,\beta)_{p,F}$.
\item[(c)] $[\alpha,\beta)_{p,F}=[\alpha+v^p-v,\beta)_{p,F}$ for any $v \in F$.
\item[(d)] For any $v \in F$ with $\beta+v^p \neq 0$, $[\alpha,\beta)_{p,F}=[\alpha',\beta+v^p)_{p,F}$ for some $\alpha' \in F$.
\end{itemize}
\end{lem}

\begin{proof}
Let $f$ be an element in $F[x]$ with $\operatorname{N}_{F[x]/F}(f) \neq 0$.
Then $(fy)^p=\operatorname{N}_{F[x]/F}(f)y^p=\operatorname{N}_{F[x]/F}(f) \beta$.
Since $(fy) x-x (fy)=fy$, we have $[\alpha,\beta)_{p,F}=[x^p-x,(fy)^p)_{p,F}=[\alpha,\operatorname{N}_{F[x]/F}(f) \beta)_{p,F}$.

Write $z=x+y$.
By \cite[Lemma 3.1]{Chapman:2015}, $z$ satisfies $z^p-z=x^p-x+y^p=\alpha+\beta$. Since $y z-z y=y$ we have $[\alpha,\beta)_{p,F}=[z^p-z,y^p)_{p,F}=[\alpha+\beta,\beta)_{p,F}$.

Write $t=x+v$. Then $t^p-t=x^p-x+v^p-v=\alpha+v^p-v$.
Since $y t-t y=y$, we have $[\alpha,\beta)_{p,F}=[\alpha+v^p-v,\beta)_{p,F}$.

Let $v \in F$ be such that $\beta+v^p \neq 0$.
Consider the element $y+v$.
Since $(y+v)^p=\beta+v^p \in F^\times$, by \citep[Chapter 7, Lemma 10]{Albert} we have $[\alpha,\beta)_{p,F}=[\alpha',\beta+v^p)_{p,F}$ for some $\alpha' \in F$.
\end{proof}

Note that by Lemma \ref{addtofirst} $(a)$, we have $[\alpha,\beta)_{p,F}=[\alpha,-\beta)_{p,F}$ for any symbol algebra $[\alpha,\beta)_{p,F}$.

\begin{lem}\label{two}
Consider $[\alpha,\beta)_{p,F}$ and $[\gamma,\delta)_{p,F}$ for some $\alpha,\gamma \in F$ and $\beta,\delta \in F^\times$.
Then 
\begin{itemize}
\item[(a)] $[\alpha,\beta)_{p,F} \otimes [\gamma,\delta)_{p,F}=[\alpha+\gamma,\beta)_{p,F} \otimes [\gamma,\beta^{-1} \delta).$
\item[(b)] If $\beta+\delta \neq 0$ then $[\alpha,\beta)_{p,F} \otimes [\gamma,\delta)_{p,F}=[\alpha+\gamma,\beta+\delta)_{p,F} \otimes C$ for some symbol algebra $C$ of degree $p$.
\end{itemize}
\end{lem}

\begin{proof}
Write $[\alpha,\beta)_{p,F}=F \langle x,y : x^p-x=\alpha, y^p=\beta, yx-xy=y \rangle$ and $[\gamma,\delta)_{p,F}=F \langle z,w : z^p-z=\gamma, w^p=\delta, wz-zw=w \rangle$.
Note that the elements $x+z$ and $y$ commute with the elements $z$ and $y^{-1} w$, and that 
$$(x+z)^p-(x+z)=x^p-x+z^p-z=\alpha+\gamma,$$
$$y (x+z)-(x+z) y=y,$$
$$(y^{-1}w)^p=y^{-p} w^p=\beta^{-1} \delta, \enspace \operatorname{and}$$
$$(y^{-1} w) z-z (y^{-1} w)=y^{-1} (w z-z w)=y^{-1} w.$$
Therefore $[\alpha,\beta)_{p,F} \otimes [\gamma,\delta)_{p,F}=F \langle x,y,z,w \rangle = F\langle x+z,y \rangle \otimes F \langle z, y^{-1} w \rangle=[\alpha+\gamma,\beta)_{p,F} \otimes [\gamma,\beta^{-1} \delta)_{p,F}$.

Now, the elements $x+z$ and $y+w$ satisfy
$$(y+w)^p=y^p+w^p=\beta+\delta, \enspace \operatorname{and}$$
$$(y+w) (x+z)-(x+z) (y+w)=y+w.$$
Therefore $F \langle x+z,y+w \rangle=[\alpha+\gamma,\beta+\delta)_{p,F}$.
Hence $[\alpha,\beta)_{p,F} \otimes [\gamma,\delta)_{p,F}=[\alpha+\gamma,\beta+\delta)_{p,F} \otimes C$ for some central simple algebra $C$ over $F$ of degree $p$.
Since $C$ contains $y^{-1}w$ which satisfies $(y^{-1}w)^p \in F^\times$, $C$ must be a symbol algebra (\cite[Chapter 7, Lemma 10]{Albert}).
\end{proof}

\section{Bounding the Symbol Length}

\begin{prop}\label{change_left}
Let $p$ be a prime integer and let $F$ be a field with $\operatorname{char}(F) = p$.
Consider the algebras $[\alpha_1,\beta_1)_{p,F},\dots,[\alpha_k,\beta_k)_{p,F}$ for some $\alpha_1,\dots,\alpha_k \in F$ and $\beta_1,\dots,\beta_k \in F^\times$ where $k$ is some positive integer.
For each $i \in \{1,\dots,k\}$, write $[\alpha_i,\beta_i)_{p,F}=F \langle x_i, y_i : x_i^p-x_i=\alpha_i, y_i^p=\beta_i, y_i x_i-x_i y_i=y_i \rangle$.

Let $\varphi$ be the homogeneous polynomial form defined on $$V=F \oplus F \oplus F[x_1] \oplus \dots \oplus F[x_k]$$ by $$\varphi(u,v,f_1,\dots,f_k)=u^p (\alpha_1+\dots+\alpha_k)-u^{p-1} v+v^p+\operatorname{N}_{F[x_1]/F}(f_1) \beta_1+\dots+\operatorname{N}_{F[x_k]/F}(f_k) \beta_k.$$

\begin{itemize}
\item[(a)] For every $(u,v,f_1,\dots,f_k) \in V$ with $u \neq 0$, we have $\bigotimes_{i=1}^k [\alpha_i,\beta_i)_{p,F}=\bigotimes_{i=1}^k C_i$ where $C_1,\dots,C_k$ are symbol algebras of degree $p$ and\\ $C_1=[\varphi(1,\frac{v}{u},\frac{f_1}{u},\dots,\frac{f_k}{u}),\beta')_{p,F}$ for some $\beta' \in F^\times$.
\item[(b)] Let $(0,v,f_1,\dots,f_k) \in V$ such that the following elements are nonzero: $f_1,\dots,f_t$ for some $t \in \{1,\dots,k\}$, $\sum_{i=1}^s \operatorname{N}_{F[x_i]/F}(f_i) \beta_i$ for any $s \in \{1,\dots,t\}$ and $(\sum_{i=1}^t \operatorname{N}_{F[x_i]/F}(f_i) \beta_i)+v^p$. Then $\bigotimes_{i=1}^k [\alpha_i,\beta_i)_{p,F}=\bigotimes_{i=1}^k C_i$ where $C_1,\dots,C_k$ are symbol algebras of degree $p$ and $C_1=[\alpha',(\sum_{i=1}^t \operatorname{N}_{F[x_i]/F}(f_i) \beta_i)+v^p)_{p,F}$ for some $\alpha' \in F$.
\item[(c)] If there exists $(u,v,f_1,\dots,f_k) \in V \setminus \{(0,\dots,0)\}$ such that $\varphi(u,v,f_1,\dots,f_k)=0$, then $\bigotimes_{i=1}^k [\alpha_i,\beta_i)_{p,F}=\bigotimes_{i=1}^k C_i$ where $C_1=M_p(F)$ and $C_2,\dots,C_k$ are symbol algebras of degree $p$.
\end{itemize}
\end{prop}

\begin{proof}
\sloppy
Let $(u,v,f_1,\dots,f_k) \in V$ with $u \neq 0$.
For each $i \in \{1,\dots,k\}$ with $\operatorname{N}_{F[x_i]/F}(\frac{f_i}{u}) \neq 0$, we apply Lemma \ref{addtofirst} to change $[\alpha_i,\beta_i)_{p,F}$ to $[\alpha_i+\operatorname{N}_{F[x_i]/F}(\frac{f_i}{u}) \beta_i,\operatorname{N}_{F[x_i]/F}(\frac{f_i}{u}) \beta_i)_{p,F}$.
Then for each $i \in \{2,\dots,k\}$, we apply Lemma \ref{two} (a) on the first and $i$th symbol algebras.
The first algebra in the tensor product obtained after these modifications is $[\sum_{i=1}^k (\alpha_i+\operatorname{N}_{F[x_i]/F}(\frac{f_i}{u}) \beta_i),\operatorname{N}_{F[x_1]/F}(\frac{f_1}{u}) \beta_1)_{p,F}$ or $[\sum_{i=1}^k (\alpha_i+\operatorname{N}_{F[x_i]/F}(\frac{f_i}{u}) \beta_i),\beta_1)_{p,F}$ depending on the value of $\operatorname{N}_{F[x_1]/F}(\frac{f_1}{u})$.
By Lemma \ref{addtofirst} $(c)$ we can add $(\frac{v}{u})^p-\frac{v}{u}$ to the left slot of the first algebra. That proves part $(a)$.
If $\varphi(u,v,f_1,\dots,f_k)=0$ then also $\varphi(1,\frac{v}{u},\frac{f_1}{u},\dots,\frac{f_k}{u})=0$, which means that the first algebra in the modified tensor product is a matrix algebra.

\sloppy
Let $(0,v,f_1,\dots,f_k) \in V$ such that the following elements are nonzero: $f_1,\dots,f_t$ for some $t \in \{1,\dots,k\}$, $\sum_{i=1}^s \operatorname{N}_{F[x_i]/F}(f_i) \beta_i$ for any $s \in \{1,\dots,t\}$ and $(\sum_{i=1}^t \operatorname{N}_{F[x_i]/F}(f_i) \beta_i)+v^p$.
For each $i \in \{1,\dots,t\}$ change $[\alpha_i,\beta_i)_{p,F}$ to $[\alpha_i,\operatorname{N}_{F[x_i]/F}(f_i) \beta_i)_{p,F}$.
Then for each $i \in \{2,\dots,t\}$ apply Lemma \ref{two} (b) on the first and the $i$th symbol algebras.
The first algebra in the resulting tensor product is $[\alpha_1+\dots+\alpha_t,\sum_{i=1}^t \operatorname{N}_{F[x_i]/F}(f_i) \beta_i)_{p,F}$. Then by Lemma \ref{addtofirst} (d) we can change this algebra to $[\alpha',(\sum_{i=1}^t \operatorname{N}_{F[x_i]/F}(f_i) \beta_i)+v^p)_{p,F}$ for some $\alpha' \in F$. That proves part $(b)$.

Let $(0,v,f_1,\dots,f_k) \in V \setminus \{(0,\dots,0)\}$ such that $\varphi(0,v,f_1,\dots,f_k)=0$.
If $f_1=\dots=f_k=0$ then $\varphi(0,v,f_1,\dots,f_k)=v^p=0$ which is impossible for $v \in F^\times$.
Therefore at least one $f_i$ is nonzero.
If there is one $f_i$ which is nonzero and $\operatorname{N}_{F[x_i]/F}(f_i)=0$ then the symbol algebra $[\alpha_i,\beta_i)_{p,F}$ contains a zero divisor and so it is a matrix algebra.
Assume every nonzero $f_i$ has $\operatorname{N}_{F[x_1]/F}(f_i) \neq 0$.
Without loss of generality we can assume $f_1,\dots,f_t$ are nonzero for some $t \in \{1,\dots,k\}$ and $f_i=0$ for $i>t$.
Assume $\sum_{i=1}^s \operatorname{N}_{F[x_i]/F}(f_i) \beta_i=0$ for some $s \in \{1,\dots,t\}$. Assume $s$ is minimal.
Then we can change the tensor product so that the first algebra is $[\alpha',\sum_{i=1}^{s-1} \operatorname{N}_{F[x_i]/F}(f_i) \beta_i)_{p,F}$ and the $s$th algebra is $[\alpha_s,\operatorname{N}_{F[x_s]/F}(f_s) \beta_s)_{p,F}$.
Since $\sum_{i=1}^{s-1} \operatorname{N}_{F[x_i]/F}(f_i) \beta_i=-\operatorname{N}_{F[x_s]/F}(f_s) \beta_s$, we have $[\alpha',\sum_{i=1}^{s-1} \operatorname{N}_{F[x_i]/F}(f_i) \beta_i)_{p,F} \otimes [\alpha_s,\operatorname{N}_{F[x_s]/F}(f_s) \beta_s)_{p,F}=[\alpha'+\alpha_s,\operatorname{N}_{F[x_s]/F}(f_s) \beta_s)_{p,F} \otimes M_p(F)$.
Assume $\sum_{i=1}^s \operatorname{N}_{F[x_i]/F}(f_i) \beta_i \neq 0$ for all $s \in \{1,\dots,t\}$.
Then the tensor product can be changed such that the first algebra is $[\alpha',\sum_{i=1}^t \operatorname{N}_{F[x_i]/F}(f_i) \beta_i)_{p,F}$.
This algebra contains a zero divisor because $\varphi(0,v,f_1,\dots,f_k)=(\sum_{i=1}^t \operatorname{N}_{F[x_i]/F}(f_i) \beta_i)+v^p=0$, which means it is a matrix algebra. That completes part $(c)$.
\end{proof}

\begin{thm}\label{linkage}
Let $p$ be a prime integer and let $F$ be a field with $\operatorname{char}(F) = p$.
Assume the maximal dimension of an anisotropic homogeneous polynomial form of degree $p$ over $F$ is a finite integer $d$.
Then every two tensor products $A=\bigotimes_{i=1}^k [\alpha_i,\beta_i)_{p,F}$ and $B=\bigotimes_{i=1}^\ell [\gamma_i,\delta_i)_{p,F}$ with $(k+\ell) p \geq d-1$ can be changed such that $\alpha_1=\gamma_1$.
\end{thm}

\begin{proof}
Let $\varphi$ and $\psi$ be the homogeneous polynomial forms of degree $p$ as constructed in Proposition \ref{change_left} for $A$ and $B$ respectively.
If $\varphi$ has a nontrivial zero then the first algebra in $A$ can be assumed to be a matrix algebra, and so it can be written as $[\gamma_1,1)_{p,F}$ and the statement follows.
Similarly, the statement follows when $\psi$ has a nontrivial zero.

\sloppy
Assume that both forms are anisotropic.
The form $\varphi(u,v,f_1,\dots,f_k)-\psi(u,0,f_1',\dots,f_\ell')$ 
is of dimension $(k+\ell)p+2$ over $F$.
By assumption, $(k+\ell) p+2 \geq d+1$, which means that this form has a nontrivial zero, i.e. $\varphi(u,v,f_1,\dots,f_k)=\psi(u,0,f_1',\dots,f_\ell')$ for some $u,v,f_1,\dots,f_k,f_1',\dots,f_\ell'$ such that not all of them are $0$.
If $u \neq 0$ then $\varphi(1,\frac{v}{u},\frac{f_1}{u},\dots,\frac{f_k}{u})=\psi(1,0,\frac{f_1'}{u},\dots,\frac{f_\ell'}{u})$.
By Proposition \ref{change_left} (a) we can change both tensor products to have $\varphi(1,\frac{v}{u},\frac{f_1}{u},\dots,\frac{f_k}{u})$ in the left slot of the first algebra.

\sloppy
Assume $u=0$.
At least one of the elements $v,f_1,\dots,f_k,f_1',\dots,f_\ell'$ is nonzero.
If $f_i'=0$ for every $i \in \{1,\dots,\ell\}$ then $\psi(0,0,f_1',\dots,f_\ell')=0$ and so $\varphi(0,v,f_1,\dots,f_k)=0$, but since $\varphi$ is anisotropic we get $v=f_1=\dots=f_k=0$, contradiction. Therefore $f_i' \neq 0$ for at least one $i$ in $\{1,\dots,\ell\}$. If $f_i=0$ for every $i \in \{1,\dots,k\}$ then $v^p=\psi(0,0,f_1',\dots,f_\ell')$, which means $0=\psi(0,0,f_1',\dots,f_\ell')-v^p=\psi(0,-v,f_1',\dots,f_\ell')$, i.e. $\psi$ has a nontrivial zero, contradiction.
Therefore $f_i \neq 0$ for at least one $i$ in $\{1,\dots,k\}$.
By changing the order of the symbol algebras, we can assume $f_1,\dots,f_t \neq 0$ for some $t \in \{1,\dots,k\}$, $f_i=0$ for every $i$ with $i \geq t+1$, $f_1',\dots,f'_r \neq 0$ for some $r \in \{1,\dots,\ell\}$ and $f_i'=0$ for every $i$ with $i \geq r+1$.
Both $(0,v,f_1,\dots,f_k)$ and $(0,0,f_1',\dots,f_
\ell')$ meet the requirements of Proposition \ref{change_left} (b), and so we can change the first tensor product so that it has $\varphi(0,v,f_1,\dots,f_k)$ in the second slot of the first symbol algebra and the second tensor product so that it has $\psi(0,0,f_1',\dots,f_\ell')$ in the second slot of the first symbol algebra.
The statement then follows from \cite[Theorem 3.2]{Chapman:2015}.
\end{proof}

\begin{cor}\label{corcharp}
Let $p$ be a prime integer and let $F$ be a field with $\operatorname{char}(F) = p$.
Assume the maximal dimension of an anisotropic homogeneous polynomial form of degree $p$ over $F$ is a finite integer $d$ greater than 1.
Then the symbol length of $p$-algebras of exponent $p$ over $F$ is bounded from above by $\left \lceil \frac{d-1}{p} \right \rceil -1$.
\end{cor}

\begin{proof}
Write $n=\left \lceil \frac{d-1}{p} \right \rceil -1$.
Then it suffices to prove that for every tensor product $A=\bigotimes_{i=1}^k [\alpha_i,\beta_i)_{p,F}$ with $k \geq n+1$, 
$A=M_p(F) \otimes C_1 \otimes \dots \otimes C_{k-1}$ for some symbol algebras $C_1,\dots,C_{k-1}$ of degree $p$ and the statement then follows by induction.

Note that if $k \geq n+1=\left \lceil \frac{d-1}{p} \right \rceil$ then $kp \geq d-1$.
Consider the algebra $[\alpha_1,\beta_1)_{p,F}$ and the tensor product $\bigotimes_{i=2}^k [\alpha_i,\beta_i)_{p,F}$.
By Theorem \ref{linkage}, since $p(1+(k-1))=kp \geq d-1$ we can assume $\alpha_1=\alpha_2$, and so 
$$[\alpha_1,\beta_1)_{p,F} \otimes [\alpha_2,\beta_2)_{p,F} \otimes \dots \otimes [\alpha_k,\beta_k)_{p,F}=M_p(F) \otimes [\alpha_1,\beta_1 \beta_2)_{p,F} \otimes [\alpha_3,\beta_3)_{p,F} \otimes \dots \otimes [\alpha_k,\beta_k)_{p,F}.$$
\end{proof}

\begin{rem}
By using similar methods, the upper bound for the symbol length of algebras of exponent $p$ over a field $F$ with $\operatorname{char}(F) \neq p$ containing primitive $p$th roots of unity appearing in \cite[Theorem 3.4]{Matzri} can be sharpened from $\left \lceil \frac{d+1}{p} \right \rceil -1$ to $\left \lceil \frac{d-1}{p} \right \rceil -1$ as well, where the maximal dimension of an anisotropic homogeneous polynomial form of degree $p$ over $F$ is $d$ greater than 1.
In the case of $p=2$ and $\operatorname{char}(F) \neq 2$, $d$ is the $u$-invariant of $F$, $u(F)$, and so this upper bound is $\left \lceil \frac{u(F)-1}{2} \right \rceil$ which appears in \cite[Theorem 2 (a)]{Kahn:1990}. When we assume further that $I^3 F=0$, this bound is sharp (see \cite[Theorem 2 (b)]{Kahn:1990}). For the construction of fields $F$ with $\operatorname{char}(F) \neq 2$, $I^3 F=0$ and any even $u(F)$ see \cite[Theorem 4]{Merkurjev1991}.
\end{rem}

\section{Quaternion Algebras}

In this section we focus on the case of $p=2$ and $\operatorname{char}(F)=2$.
In this case, the symbol algebras are quaternion algebras.
The corresponding homogeneous polynomial forms of degree $p$ constructed in the previous section are now quadratic forms.
Quadratic forms in this case are of the shape
$$[a_1,b_1] \perp \dots \perp [a_r,b_r] \perp \langle c_1,\dots,c_t \rangle$$
for some $a_1,b_1,\dots,a_r,b_r,c_1,\dots,c_t \in F$.
Each $[a_i,b_i]$ stands for the quadratic form $a_i u_i^2+u_i v_i+b_i v_i^2$, $\langle c_1,\dots,c_t \rangle$ is the diagonal form $c_1 w_1^2+\dots+c_t w_t^2$, and $\perp$ is the orthogonal sum.
A quadratic form is nonsingular if $t=0$.
In particular, the dimension of a nonsingular quadratic form is always even.
Note that $c [a,b]$ is isometric to $[\frac{a}{c},bc]$ for $c \in F^\times$, so orthogonal sums of quadratic forms of the shape $c [a,b]$ are also nonsingular.

Given $K=F[x : x^2+x=\alpha]$, the norm form $\norm_{K/F} : K \rightarrow F$ mentioned in Remark \ref{rem1} $(e)$ is the quadratic form $[\alpha,1]$.

The maximal dimension of an anisotropic quadratic form over $F$ is denoted by $\hat{u}(F)$. The number $d$ appearing in Theorem \ref{linkage} and Corollary \ref{corcharp} is equal to $\hat{u}(F)$ when $p=2$.
The $u$-invariant of $F$, denoted by $u(F)$, is the maximal dimension of an anisotropic nonsingular quadratic form over $F$.
Clearly $u(F) \leq \hat{u}(F)$, and there are many examples in the literature where this inequality is strict (see for example \cite{MammoneTignolWadsworth:1991}).
Therefore, using $u(F)$ instead of $\hat{u}(F)$ gives a better upper bound for the symbol length in the case of $p=2$.
We now rephrase Theorem \ref{linkage} and Corollary \ref{corcharp} in terms of $u(F)$.

\begin{thm}
Let $F$ be a field with $\operatorname{char}(F) = 2$ and $u(F) < \infty$.
Then every two tensor products of quaternion algebras $A=\bigotimes_{i=1}^k [\alpha_i,\beta_i)_{2,F}$ and $B=\bigotimes_{i=1}^\ell [\gamma_i,\delta_i)_{2,F}$ with $(k+\ell) 2 \geq u(F)$ can be changed such that $\alpha_1=\gamma_1$.
\end{thm}

\begin{proof}
The proof is essentially the same as in Theorem \ref{linkage}.
One just has to note that the form 
$\varphi(u,v,f_1,\dots,f_k)-\psi(u,0,f_1',\dots,f_\ell')$ in this case is the quadratic form
$$[\alpha_1+\dots+\alpha_k+\gamma_1+\dots+\gamma_\ell,1] \perp \beta_1 [\alpha_1,1] \perp \dots \perp \beta_k [\alpha_k,1] \perp \delta_1 [\gamma_1,1] \perp \dots \perp \delta_\ell [\gamma_\ell,1].$$
This quadratic form is nonsingular of dimension $(k+\ell) 2+2$, which is greater than $u(F)$. Hence it has a nontrivial zero, and the proof continues in the same manner as the other proof.
\end{proof}

\begin{cor}\label{corchar2}
Let $F$ be a field with $\operatorname{char}(F) = 2$ and $2 \leq u(F) < \infty$.
Then the symbol length of algebras of exponent $2$ over $F$ is bounded from above by $\frac{u(F)}{2} -1$.
\end{cor}

In certain cases this upper bound is sharp as Proposition \ref{remchar2} suggests. Before stating this proposition, we recall a few facts about quadratic forms: The unique (up to isometry) isotropic nonsingular quadratic form of dimension 2 over $F$ is the hyperbolic plane $\varmathbb{H}=[0,1]$. Every nonsingular quadratic form $\varphi$ over $F$ can be written uniquely as $\varphi_{\operatorname{an}} \perp \underbrace{\varmathbb{H} \perp \dots \perp \varmathbb{H}}_{m \enspace \operatorname{times}}$ for some anisotropic form $\varphi_{\operatorname{an}}$ and some nonnegative integer $m$ called the Witt index of $\varphi$. Two forms $\varphi$ and $\phi$ are Witt equivalent if their underlying anisotropic subforms $\varphi_{\operatorname{an}}$ and $\phi_{\operatorname{an}}$ are isometric. The group of Witt equivalence classes of nonsingular quadratic forms over $F$ with $\perp$ as the group operation is denoted by $I_q F$.
The ``Arf invariant" or ``discriminant" of a nonsingular form $\varphi=[a_1,b_1] \perp \dots \perp [a_r,b_r]$ is the class of $\sum_{i=1}^r a_i b_i$ in the additive group $F/\wp(F)$ where $\wp(F)=\{\lambda^2+\lambda : \lambda \in F\}$. The subgroup of $I_q F$ of forms with trivial discriminant is $I_q^2 F$. 
The Clifford invariant maps $2r$-dimensional nonsingular forms $\varphi=[a_1,b_1] \perp \dots \perp [a_r,b_r]$ with trivial discriminant to tensor products of $r-1$ quaternion algebras $$E(\varphi)=[a_1 b_1,b_r b_1)_{2,F} \otimes \dots \otimes [a_{r-1} b_{r-1},b_r b_{r-1})_{2,F}.$$ For every tensor product $T$ of $r-1$ quaternion algebras there exists a nonsingular form $\varphi$ of dimension $2r$ with trivial discriminant such that $E(\varphi)=T$ (see \cite[Proof of Theorem 4.1]{Chapman:2015:chain}).
The Clifford invariant defines a group epimorphism from $I_q^2 F$ to $\prescript{}{2} Br(F)$ whose kernel is exactly $I_q^3 F$.

\begin{lem}\label{addedlem}
If $\operatorname{char}(F)=2$, $I_q^3 F=0$ and $4 \leq u(F) < \infty$ then there exists an anisotropic nonsingular quadratic form with trivial discriminant of dimension $u(F)$.
\end{lem}

\begin{proof}
Let $\phi$ be an anisotropic nonsingular form of dimension $u(F)$. Denote its discriminant by $\delta$.
If $\delta \in \wp(F)$, we are done, so assume $\delta \not \in \wp(F)$.
Write $\psi=\phi \perp [\delta,1]$.
Since $\psi$ is of dimension $u(F)+2$, it is isotropic.
The Witt index of $\psi$ can be either $1$ or $2$.
If it is $2$ then $\psi=\psi_{\operatorname{an}} \perp \varmathbb{H} \perp \varmathbb{H}=\psi_{\operatorname{an}} \perp [\delta,1] \perp [\delta,1]$, and so $\phi=\psi_{\operatorname{an}} \perp [\delta,1]$.
Note that the discriminant of $\psi_{\operatorname{an}}$ is trivial. Since $I_q^3 F=0$, $\psi_{\operatorname{an}}$ is universal (see \cite[Proof of Theorem 4.1]{BarryChapman:2015}), which contradicts the assumption that $\phi$ is anisotropic.
Consequently the Witt index of $\psi$ must be 1.
Therefore $\psi_{\operatorname{an}}$ is an anisotropic nonsingular form of dimension $u(F)$ with trivial discriminant.
\end{proof}

\begin{rem}
By \cite[Theorem 2]{Kahn:1990} $(b)$, $u(F)$ is even when $\operatorname{char}(F) \neq 2$, $I^3 F=0$ and $u(F) > 1$. One can therefore prove in a similar way to Lemma \ref{addedlem} that if $\operatorname{char}(F) \neq 2$, $I^3 F=0$ and $4 \leq u(F) < \infty$ then there exists an anisotropic quadratic form with trivial discriminant of dimension $u(F)$.
\end{rem}

\begin{prop}\label{remchar2}
When $\operatorname{char}(F)=2$, $I_q^3 F=0$ and $2 \leq u(F) < \infty$, the bound $\frac{u(F)}{2}-1$ for the symbol length of algebras of exponent 2 over $F$ is sharp.
\end{prop}

\begin{proof}
Write $u(F)=2 n$ for some positive integer $n$. The statement is clearly true when $n=1$, so assume $n \geq 2$. It is enough to find an algebra of exponent 2 over $F$ whose symbol length is $\frac{u(F)}{2}-1=n-1$. 
By Lemma \ref{addedlem} there exists an anisotropic nonsingular form $\varphi$ with trivial discriminant of dimension $2n$.
We claim that the symbol length of $E(\varphi)$ is $n-1$.
Clearly it is no greater than $n-1$.
Suppose $E(\varphi)$ is Brauer equivalent to a tensor product $T$ of $k$ quaternion algebras where $k < n-1$.
Then there exists some nonsingular form $\phi$ of dimension $2(k+1)$ with trivial discriminant such that $E(\phi)=T$.
Since $E(\phi)$ and $E(\varphi)$ are Brauer equivalent, the forms $\phi$ and $\varphi$ are equivalent modulo $I_q^3 F$. However, $I_q^3 F=0$, which means that $\phi$ and $\varphi$ are Witt equivalent, but this is impossible because $\varphi$ is anisotropic of dimension $2n$ and $\phi$ is of dimension $2(k+1) < 2n$.
\end{proof}

Corollary \ref{corchar2} and Proposition \ref{remchar2} provide characteristic 2 analogues to parts $(a)$ and $(b)$ of \cite[Theorem 2]{Kahn:1990}.
For the construction of fields $F$ with $\operatorname{char}(F)=2$, $I_q^3 F=0$ and any even $u$-invariant see \cite[Theorem 38.4]{EKM}.

\begin{rem}
There is a typographical error in the first sentence, second paragraph of the proof of \cite[Theorem 4.1]{BarryChapman:2015}. The correct sentence should be: ``Now we show that if $f \in I_q^2 F$ is an anisotropic form of dimension at least $6$ and $f_K$ is isotropic for some $K=F[x: x^2+x=a]$ then $f_K=\varmathbb{H} \perp f_0$ where $\varmathbb{H}$ is a hyperbolic plane and $f_0$ is anisotropic".
\end{rem}

\section*{Acknowledgments}

The author thanks the anonymous referee for the helpful comments on the manuscript.

\section*{Bibliography}
\bibliographystyle{amsalpha}
\bibliography{bibfile}
\end{document}